\RequirePackage[english]{babel}
\documentclass[a4paper,twoside,11pt]{amsart}
\usepackage[latin1]{inputenc}
\usepackage{amsmath}
\usepackage{amssymb}
\usepackage{amsfonts}
\usepackage{mathrsfs}
\usepackage{hyperref}
\usepackage{times}

\usepackage{enumerate}

\ifpdf
  \usepackage[matrix,arrow,cmtip,2cell]{xy}
\else
  \usepackage[matrix,arrow,cmtip,2cell,dvips]{xy}
\fi

\SelectTips{cm}{}
\newdir{(}{{}*!/-5pt/@^{(}}
\newdir{(x}{{}*!/-5pt/@_{(}}
\newdir{+}{{}*!/-9pt/{}}
\newdir{>+}{@{>}*!/-9pt/{}}
\entrymodifiers={+!!<0pt,\fontdimen22\textfont2>}
\UseAllTwocells

\usepackage[nosubsections]{mytheorems2}

\newtheoremstyle{dtheoremnopar}{3 mm}{1 mm}{\itshape}{}{\bfseries}{.}{ }
{\thmname{#1}\thmnumber{ #2}\thmnote{ \mdseries(#3)\bfseries}}

\theoremstyle{dtheoremnopar}
\newcounter{theoremx}
\newtheorem{theoremalpha}[theoremx]{Theorem}

\newcommand{\tref}[1]{\ref{#1}} 

\newcommand{\pref}[1]{\eqref{#1}}
\newcommand\inj{\hookrightarrow}
\newcommand\surj{\twoheadrightarrow}
\newcommand\map[3]{#1\colon #2\rightarrow #3}
\newcommand\injmap[3]{#1\colon #2\hookrightarrow #3}
\newcommand\id[1]{\mathrm{id}_{#1}} 
\DeclareMathOperator{\Hom}{Hom}
\DeclareMathOperator{\End}{End}
\DeclareMathOperator{\Ann}{Ann}
\newcommand\sA{\mathcal{A}}
\newcommand\sF{\mathcal{F}}
\newcommand\sG{\mathcal{G}}
\newcommand\sK{\mathcal{K}}
\newcommand\sO{\mathcal{O}}
\DeclareMathOperator{\Spec}{Spec}
\DeclareMathOperator{\Supp}{Supp}  
\newcommand{\etale}{\'{e}tale}
\newcommand\Quot{\mathrm{Quot}}
\newcommand\FHilb{\mathcal{H}\mathit{ilb}}
\newcommand\Sch{\mathbf{Sch}}
\newcommand\QCoh{\mathbf{QCoh}} 

\newcommand\HilbSt{\mathscr{H}}

\newcommand\unram{\mathrm{unram}}
\newcommand\qfin{\mathrm{qfin}}

\newcommand\Sec{\mathbf{Sec}}
\newcommand\weilr{\mathbf{R}} 
\newcommand\Coh{\mathscr{C}oh} 
\newcommand\CohAlg{\mathscr{C}oh^{\mathrm{alg}}} 
\newcommand\stG{\mathscr{G}} 

\newcommand{\spref}[1]{\href{http://stacks.math.columbia.edu/tag/#1}{#1}}

\begin{document}

\title{General Hilbert stacks and Quot schemes}
\author{Jack Hall}
\address{Centre for Mathematics and its Applications\\
        Mathematical Sciences Institute\\
        John Dedman Building\\
        The Australian National University\\
        Canberra, ACT 2601\\
        Australia}
\email{jack.hall@anu.edu.au}
\author{David Rydh}
\address{KTH Royal Institute of Technology\\
         Department of Mathematics\\
         SE\nobreakdash-100\ 44\ Stockholm\\
         Sweden}
\email{dary@math.kth.se}
\thanks{This collaboration was supported by the G\"oran Gustafsson foundation.
The second author is also supported by the Swedish Research Council 2011-5599.}
\date{2015-09-30}
\dedicatory{In memory of Dan Laksov.}

\begin{abstract}
We prove the algebraicity of the Hilbert functor, the Hilbert stack, the
Quot functor and the stack of coherent sheaves on an algebraic stack $X$
with (quasi-)finite diagonal without any finiteness assumptions on $X$. We
also give similar results for Hom stacks and Weil restrictions.
\end{abstract}

\maketitle


\setcounter{secnumdepth}{0}
\begin{section}{Introduction}
Let $S$ be a scheme and let $\map{f}{X}{S}$ be a
morphism between algebraic stacks that is locally of finite presentation. If
$f$ is separated, then it is well-known that the Hilbert functor $\FHilb_{X/S}$
is an algebraic space, locally of finite presentation
over~$S$~\cite{artin_alg_formal_moduli_I,olsson-starr_quot-functors,olsson_proper-coverings}. If $f$ is not separated but has
quasi-compact and separated diagonal with affine stabilizers, then one can instead prove that the
Hilbert stack $\HilbSt^{\qfin}_{X/S}$---parameterizing proper flat families
with a quasi-finite morphism to~$X$---is an algebraic
stack, locally of finite presentation over
$S$~\cite{hall-rydh_hilbert-stack,rydh_hilbert}.
The first main result of this
paper is a partial generalization of these two results to stacks that are not
locally of finite presentation.

\begin{theoremalpha}\label{T:HILB}
Let $S$ be a scheme and let $X$ be an algebraic stack over $S$.
\begin{enumerate}
\item If $X\to S$ has finite diagonal, then $\FHilb_{X/S}$ is a
  separated algebraic space and $\HilbSt^{\qfin}_{X/S}$ is an algebraic stack
  with affine diagonal.
\item\label{T:HILB:HILBST} If $X\to S$ has quasi-compact and separated diagonal with affine stabilizers, then
  $\HilbSt^{\qfin}_{X/S}$ is an algebraic stack with quasi-affine diagonal.
\end{enumerate}
In particular, if $X$ is any separated scheme, algebraic
space or Deligne--Mumford stack, then $\FHilb_{X/S}$ is an algebraic space.
\end{theoremalpha}

Our second result is about stacks of sheaves. Let us again first recall the
classical situation. So, let $\map{f}{X}{S}$ be a separated morphism between
algebraic stacks that is locally of finite presentation. Then $\Coh(X/S)$---the
stack of finitely presented sheaves on $X$ that are flat and proper over
$S$---is an algebraic stack, locally of finite presentation over $S$ with
affine diagonal~\cite[Thm.~2.1]{lieblich_Coh_stack},
\cite[Thm.~8.1]{hall_coherent-versality}. If we are also given a quasi-coherent
sheaf $\sF$ on $X$, then $\Quot(X/S,\sF)$ is a separated algebraic
space~\cite[Cor.~8.2]{hall_coherent-versality}.
Usually, one also assumes that $\sF$ is finitely presented.
Then $\Quot(X/S,\sF)$ is locally of
finite presentation over~$S$ and the result goes back
to~\cite{artin_alg_formal_moduli_I,olsson-starr_quot-functors, olsson_proper-coverings}. Again,
we are able to remove the hypothesis that $X\to S$ is locally of finite
presentation.

\begin{theoremalpha}\label{T:COH}
Let $X$ be an algebraic stack with finite diagonal over $S$. Then the stack
$\Coh(X/S)$ is algebraic with affine diagonal. If $\sF$ is a
quasi-coherent $\sO_X$-module, then $\Quot(X/S,\sF)$ is a separated algebraic
space over $S$.
\end{theoremalpha}

When $X\to S$ is not locally of finite presentation, then the definitions of
$\Coh(X/S)$ and $\Quot(X/S,\sF)$ are somewhat subtle. The objects that are
para\-metrized are quasi-coherent sheaves $\sG$ that are flat,
\emph{intrinsically of
  finite presentation} and \emph{intrinsically proper} over $S$ (together with
a surjective homomorphism $\sF\to \sG$ for $\Quot$). Sheaves intrinsically of
finite presentation over $S$ are of finite type as $\sO_X$-modules. They are,
however, not necessarily of finite presentation as $\sO_X$-modules,
and not every finitely presented $\sO_X$-module is intrinsically of finite
presentation.


There are two key ingredients in the proofs. The first is the approximation
result~\cite[Thm.~D]{rydh_noetherian-approx}: every algebraic stack with
quasi-finite diagonal can be approximated by algebraic
stacks of finite presentation. The second is the representability result~\cite[Thm.~D]{hall_coho-bc}: if $X\to S$ is separated and locally of
finite presentation, and given $\sF,\sG\in\QCoh(X)$ such that $\sG$ is of finite
presentation, flat over $S$ and with support proper over $S$, then
$\Hom_{\sO_X/S}(\sF,\sG)$ is \emph{affine} over~$S$.

The first result shows that the morphisms $\map{f}{X}{S}$ appearing in the main theorems can be
factored as $X \to X_0 \to S$, where $X \to X_0$ is affine and $X_0 \to S$ is of finite presentation. If $P_{X/S}$ is one of the stacks figuring in the main theorems, then we will describe natural morphisms $P_{X/S} \to P_{X_0/S}$. The second result will show that these morphisms are affine.

Independently, Di Brino used similar methods to prove that $\Quot(\sF)$ is a
scheme when $\sF$ is a quasi-coherent sheaf on a projective
scheme~\cite{dibrino_quot-functor}. Approximation
for the Quot and Hilbert functors is somewhat complicated, since a
homomorphism $\sF_\lambda\to \sF$ only gives rise to a rational map
$\Quot(\sF)\dashrightarrow \Quot(\sF_\lambda)$. We apply the approximation step
to $\Hom(\sF,\sG)$, $\Coh(X/S)$ and the Hilbert stack $\HilbSt_{X/S}$ where
this inconvenience is absent. The algebraicity of $\Quot(\sF)$ and
$\FHilb_{X/S}$ then follows from the algebraicity of $\Hom(\sF,\sG)$,
$\Coh(X/S)$ and $\HilbSt_{X/S}$.
For zero-dimensional families, our results have appeared in~\cite{rydh_hilbert}
and~\cite{GLS_Affine_Hilb,GLS_Affine_Quot,skjelnes_Weil-restriction-of-mod}
using \etale{} localization
and explicit equations in the affine case.

The last two decades have witnessed an increased interest in the usage of
objects that are not of finite type---particularly in non-archimedean and
arithmetic geometry.
That being said, this paper was not written with a particular application
in mind. Rather, it was the startling realization that recent techniques
implied the existence of parameter spaces in such a great generality---in
contrast to the preconceptions of the authors---that led to this paper.

\begin{subsection}{Acknowledgments}
We would like to thank G. Di Brino and R.\ Skjelnes for useful conversations.
We are also very grateful to the referee for carefully reading the paper
and making us aware of Remark~\pref{R:weilr-smoothness}.
\end{subsection}
\end{section}
\setcounter{secnumdepth}{3}


\begin{section}{Approximation}
Let $S$ be an affine scheme (or more generally a pseudo-noetherian stack).
Recall that an algebraic stack $X\to S$ has an \emph{approximation} if there
exists a factorization $X\to X_0\to S$ where $X\to X_0$ is affine and $X_0\to
S$ is of finite presentation~\cite[Def.~7.1]{rydh_noetherian-approx}.
Equivalently, there is an inverse system $\{X_\lambda\}$ of algebraic stacks of
finite presentation over $S$ with affine bonding maps and inverse limit
$X$~\cite[Prop.~7.3]{rydh_noetherian-approx}.

We say that a morphism $X\to S$ is \emph{locally of approximation type} if
there exist a faithfully flat morphism $S'\to S$ that is locally of finite
presentation, and an \etale{} representable surjective morphism $X'\to
X\times_S S'$ such that $X'\to S'$ is a composition of a finite number of
morphisms that are either affine or locally of finite presentation
and quasi-separated.

The condition of being locally of approximation type is clearly (i) stable
under base change, (ii) stable under precomposition with morphisms that are
either affine or locally of finite presentation and quasi-separated, (iii)
fppf-local on the base and (iv) \etale{}-local on the source.

\begin{lemma}
Let $\map{f}{X}{S}$ be a quasi-compact and quasi-separated morphism of
algebraic stacks. The following are equivalent
\begin{enumerate}
\item\label{LI:approx:1}
  $f$ is locally of approximation type.
\item\label{LI:approx:2}
  There exists an fppf-covering $\{S_i\to S\}$ such that $S_i$ is
  affine and $X_i=X\times_S S_i\to S_i$ has an approximation.
\end{enumerate}
\end{lemma}
\begin{proof}
Clearly~\ref{LI:approx:2} implies~\ref{LI:approx:1}. For the converse,
we may assume that $S$ is affine
and that there exists an \etale{} representable surjective morphism $X'\to X$
such that $X'\to S$ is a composition of morphisms that are either affine or
locally of finite presentation and quasi-separated. Since $X$ is quasi-compact, we may further
assume that these morphisms are quasi-compact. Then $X'\to X$ is of
finite presentation and $X\to S$ is of approximation
type~\cite[Def.~2.9]{rydh_noetherian-approx}. It then has an approximation
by~\cite[Thm.~7.10]{rydh_noetherian-approx}.
\end{proof}

\end{section}


\begin{section}{Stacks of spaces}
In this section we prove Theorem~\tref{T:HILB} and some related algebraicity
results for $\Hom$-stacks and Weil restrictions.

\begin{definition}
Let $\map{f}{X}{S}$ be a morphism of algebraic stacks. The Hilbert stack
$\HilbSt_{X/S}$ is the category where:
\begin{itemize}
\item objects are pairs of morphisms
$({\map{p}{Z}{T}},\map{q}{Z}{X})$, where $T$ is an $S$-scheme, such that
$p$ is flat, proper and of finite presentation and the induced morphism
$\map{(q,p)}{Z}{X\times_S T}$
is representable;
\item morphisms are triples $(\varphi,\psi,\tau)$ fitting into a $2$-commutative
diagram
\vspace{-6 mm}
\[\xymatrix{
{Z_1}\ar[r]^{\varphi}\ar[d]_{p_1}\rruppertwocell<9>^{q_1}{<-2.8>\tau}
  & {Z_2}\ar[d]^{p_2}\ar[r]^{q_2} & X \\
{T_1}\ar[r]^{\psi} & {T_2}\ar@{}[ul]|\square
}\vspace{-2 mm}\]
such that the square is cartesian.
\end{itemize}
The category $\HilbSt_{X/S}$ is fibered in groupoids over $\Sch_{/S}$ and by
\etale{} descent it follows that $\HilbSt_{X/S}$ is a stack. We call
$\HilbSt_{X/S}$ the \emph{Hilbert stack} of~$X$.
\end{definition}

The substack of objects such that $\map{(q,p)}{Z}{X\times_S T}$ is
locally quasi-finite (resp.\ unramified, resp.\ a closed immersion) is denoted
$\HilbSt^{\qfin}_{X}$ (resp.\ $\HilbSt^{\unram}_{X}$,
resp.\ $\FHilb_{X/S}$). The first two substacks are always open substacks
and the third substack---the Hilbert functor---is an open substack if
$X\to S$ is separated~\cite[Prop.~1.9]{rydh_hilbert}. When $X\to S$ has
quasi-compact and separated diagonal, then $(q,p)$ is quasi-finite and
separated for every object of $\HilbSt^{\qfin}_{X/S}$. Thus, the stack
$\HilbSt^{\qfin}_{X/S}$ coincides with the Hilbert stack figuring
in~\cite{hall-rydh_hilbert-stack}.

\begin{theorem}\label{T:Hilb-stack}
Let $\map{f}{X}{S}$ be a morphism of algebraic stacks that is quasi-separated and locally of approximation type.
If $f$ is separated (resp.\ has quasi-finite and separated diagonal), then
$\HilbSt^{\qfin}_{X/S}$ is an algebraic stack with affine (resp.\ quasi-affine)
diagonal.
\end{theorem}

Theorem~\tref{T:HILB} is a consequence of Theorem~\pref{T:Hilb-stack} and the
following two facts:
\begin{enumerate}
\item an algebraic stack with quasi-finite and separated diagonal is locally of
  approximation type~\cite[Thm.~D]{rydh_noetherian-approx}; and
\item if $X$ has affine stabilizers, then
  $\HilbSt^{\qfin}_{X/S}=\HilbSt^{\qfin}_{X^\qfin/S}$ where $X^\qfin\subseteq X$
  denotes the open locus where $X$ has finite stabilizers~\cite[Pf.\ of
    Thm.~4.3]{hall-rydh_hilbert-stack}.
\end{enumerate}

Before we prove Theorem~\pref{T:Hilb-stack}, we will state a result on the
algebraicity of Weil restrictions.

\begin{theorem}\label{T:Weilr}
Let $\map{f}{Z}{S}$ be a proper and flat morphism of finite presentation
between algebraic stacks. Let $\map{g}{W}{Z}$ be a morphism of algebraic stacks
and let $f_*W=\weilr_{Z/S}(W)=\Sec_{Z/S}(W/Z)$ be the Weil restriction of $W$
along $f$.
\begin{enumerate}
\item\label{TI:Weilr-affine}
  If $\map{g}{W}{Z}$ is affine, then $f_*W\to S$ is affine.
\item\label{TI:Weilr-qaffine}
  If $\map{g}{W}{Z}$ is quasi-affine, then $f_*W\to S$ is quasi-affine.
\item\label{TI:Weilr-qc-open-imm}
  If $\map{g}{W}{Z}$ is a quasi-compact open immersion, then so is $f_*W\to S$.
\item\label{TI:Weilr-closed-imm}
  If $\map{g}{W}{Z}$ is a closed immersion, then so is $f_*W\to S$.
\item\label{TI:Weilr-findiag}
  If $\map{f}{Z}{S}$ has finite diagonal and $\map{g}{W}{Z}$ has finite
  diagonal, then
  $f_*W\to S$ is algebraic with affine diagonal.
\item\label{TI:Weilr-qfsepdiag}
  If $\map{f}{Z}{S}$ has finite diagonal and $\map{g}{W}{Z}$ has
  quasi-finite and separated diagonal, then $f_*W\to
  S$ is algebraic with quasi-affine diagonal.
\end{enumerate}
\end{theorem}

\begin{proof}[Proof of Theorems~\pref{T:Hilb-stack} and \pref{T:Weilr}]
We begin with Theorem~\pref{T:Weilr}~\ref{TI:Weilr-affine}.
Let $\map{g}{W}{Z}$ be affine. Recall that the functor
$\Hom_{\sO_Z/S}(g_*\sO_W,\sO_Z)$ is affine over $S$~\cite[Thm.~D]{hall_coho-bc}. There is a
functor $f_*W\to \Hom_{\sO_Z/S}(g_*\sO_W,\sO_Z)$ taking a section $\map{s}{Z}{W}$
of $g$ to the corresponding $\sO_Z$-module homomorphism. This functor is
represented by closed immersions. To see this, let
$\map{\varphi}{g_*\sO_W}{\sO_Z}$ be an $\sO_Z$-module homomorphism. This gives
a section of $\map{g}{W}{Z}$ if and only if the following maps vanish:
\begin{align*}
\id{\sO_Z}  - \varphi \circ \eta \colon
  & \sO_Z \to \sO_Z \\
\varphi\circ \mu - \varphi \otimes \varphi \colon
  & g_*\sO_W\otimes_{\sO_Z} g_*\sO_W \to \sO_Z,
\end{align*}
where $\map{\eta}{\sO_Z}{g_*\sO_W}$ is the unit homomorphism and
$\mu$ defines the multiplication on $g_*\sO_W$.
These conditions are closed since $\Hom_{\sO_Z/S}(\sF,\sO_Z)$ is affine, and
hence separated, for all quasi-coherent $\sO_Z$-modules
$\sF$~\cite[Thm.~D]{hall_coho-bc}.

For Theorem~\pref{T:Weilr}~\ref{TI:Weilr-qaffine}, the question easily reduces
to \ref{TI:Weilr-qc-open-imm}: if $W\to Z$ is a quasi-compact open immersion
then so is
$f_*W\to S$. Since $f_*W=S\setminus f(Z\setminus W)$ is open and constructible,
it is quasi-compact and open.

For Theorem~\pref{T:Weilr}~\ref{TI:Weilr-closed-imm}, we first assume that
$\map{g}{W}{Z}$ is a closed immersion of finite presentation. Then $f_*W\to S$
is affine, of finite presentation~\cite[Thm.~D]{hall_coho-bc} and a
monomorphism.  To show that $f_*W\to S$ is a closed immersion, it is thus
enough to verify the valuative criterion for properness. This is readily
verified since if $S$ is the spectrum of a valuation ring with generic point
$\xi$, then $Z_\xi$ is schematically dense in $Z$ by flatness of $Z\to S$.

Now suppose that $W\to Z$ is a closed immersion merely of finite type. Working
locally on $S$, we may assume that $S$ is affine. Then $W\to Z$ can be written
as an inverse limit $W=\varprojlim W_\lambda$ of finitely presented closed
immersions $W_\lambda\to Z$. It follows that $f_*W=\varprojlim f_*W_\lambda$ is
a closed immersion.

Now, we prove Theorem~\pref{T:Hilb-stack}. The question is fppf-local on $S$, so
we may assume that $S$ is affine.
Given a representable morphism $\map{g}{X}{Y}$ of algebraic stacks over $S$,
there is a natural functor
$\map{g_*}{\HilbSt_{X/S}}{\HilbSt_{Y/S}}$ taking $Z\to X$ to $Z\to X\to Y$.
If $Z\to X\to Y$ is a representable
morphism, then so is $Z\to X$.

Now assume that $g$ is quasi-affine. If $Y\to S$ is separated or has quasi-finite and separated diagonal, then we obtain an induced morphism
$\map{g_*}{\HilbSt^{\qfin}_{X/S}}{\HilbSt^{\qfin}_{Y/S}}$. Indeed, let $Z\to S$
be a proper morphism together with a quasi-finite $S$-morphism $Z\to X$. If
$Y\to S$ is separated, then $Z\to X$ is finite so that $Z\to X\to Y$ is proper
and quasi-affine, hence finite. If instead $Y\to S$ has quasi-finite and separated
diagonal, then $Z\to X\to Y$ is quasi-affine, of finite type and has proper
fibers. The last fact follows from the observation that the residual gerbe
$\stG_y$ is separated for every $y\in |Y|$. It follows that
$Z\to Y$ is quasi-finite so $g_*$ is well-defined.

Also, if $g$ is quasi-affine (resp.\ affine, resp.\ a quasi-compact
open immersion), we note
that
$\map{g_*}{\HilbSt_{X/S}}{\HilbSt_{Y/S}}$ and
$\map{g_*}{\HilbSt^{\qfin}_{X/S}}{\HilbSt^{\qfin}_{Y/S}}$
are quasi-affine (resp.\ affine, resp.\ quasi-compact open immersions).
Indeed, given a morphism $T\to \HilbSt^{\qfin}_{Y/S}$ corresponding
to maps $Z\to T$ and $Z\to Y$, then
the pull-back of $g_*$ to $T$ is $\weilr_{Z/T}(X\times_Y Z/Z)$ which is
quasi-affine (resp.\ affine, resp.\ a quasi-compact open immersion), by
Theorem~\pref{T:Weilr}~\ref{TI:Weilr-affine}--\ref{TI:Weilr-qc-open-imm}.

It is now readily deduced that $\HilbSt^{\qfin}_{X/S}=\bigcup_U
\HilbSt^{\qfin}_{U/S}$, where the union is over all open quasi-compact substacks
$U\subseteq X$. We can thus assume that $X\to S$ is quasi-compact. As the
question of algebraicity is fppf-local on $S$, we can also assume that $X\to S$
has an approximation $X \to X_0 \to S$. If $X\to S$ is separated (resp.\ has
quasi-finite and separated diagonal), then it can be arranged so that
$X_0\to S$ is also separated (resp.\ has quasi-finite and separated
diagonal)~\cite[Thm.~C]{rydh_noetherian-approx}. The stack
$\HilbSt^{\qfin}_{X_0/S}$ is thus
algebraic and has affine (resp.\ quasi-affine)
diagonal~\cite[Thm.~9.1]{hall_coherent-versality}
and~\cite[Thm.~2]{hall-rydh_hilbert-stack}. As we have seen above, the morphism
$\HilbSt^{\qfin}_{X/S} \to \HilbSt^{\qfin}_{X_0/S}$ is affine.
This proves Theorem~\pref{T:Hilb-stack}.

Parts \ref{TI:Weilr-findiag}--\ref{TI:Weilr-qfsepdiag} of
Theorem~\pref{T:Weilr} follow from Theorem~\pref{T:Hilb-stack} since the
morphism $\weilr_{Z/S}(W/Z)\to \HilbSt^{\qfin}_{W/S}$,
taking a section to its graph, is an open immersion.
\end{proof}

\begin{corollary}\label{C:homstacks}
Let $\map{f}{Z}{S}$ be a proper and flat morphism of finite presentation
between algebraic stacks and let $\map{g}{X}{S}$ be a morphism of algebraic
stacks.
\begin{enumerate}
\item If $Z\to S$ has finite diagonal and $X\to S$ has quasi-finite and
  separated diagonal, then $\Hom_S(Z,X)$ is an algebraic stack with quasi-affine
  diagonal.
\item If $X\to S$ has finite diagonal, then $\Hom_S(Z,X)$ is an algebraic stack
  with affine diagonal.
\end{enumerate}
\end{corollary}
\begin{proof}
Note that there is an isomorphism $\Hom_S(Z,X)\to \weilr_{Z/S}(X\times_S Z/Z)$
taking a morphism $\map{h}{Z}{X}$ to the section $\map{(h,\id{Z})}{Z}{X\times_S
  Z}$. Thus, in the first case the corollary follows immediately from
Theorem~\pref{T:Weilr}. In the second case, we reduce as before to the case
where $S$ is affine and $X\to S$ is quasi-compact so there is an
approximation $X\to X_0\to S$. As before, $\Hom_S(Z,X)\to \Hom_S(Z,X_0)$ is affine
since for any $T\to \Hom_S(Z,X_0)$ the pull-back is the Weil restriction
$\weilr_{Z\times_S T/T}(X\times_{X_0} Z\times_S T)$.
Finally, $\Hom_S(Z,X_0)$ is algebraic with
affine diagonal by~\cite[Cor.~9.2]{hall_coherent-versality}.
\end{proof}

\begin{remark}
Let $Z\to S$ be as in Theorem~\pref{T:Weilr} and let $\map{h}{W_1}{W_2}$ be a
morphism between stacks over $Z$ such that either $W_i\to Z$ are quasi-affine
or $Z\to S$ and $W_i\to Z$ have quasi-finite and separated diagonals.  Then
$f_*W_1$ and $f_*W_2$ are algebraic by Theorem~\pref{T:Weilr}. If $h$ has one
of the properties: affine, quasi-affine, closed immersion, open immersion,
quasi-compact open immersion, monomorphism,
Deligne--Mumford, representable, representable and separated,
locally of finite presentation, locally of finite type, unramified, \etale;
then so has $f_*W_1\to f_*W_2$. These are routine verifications. Indeed,
first reduce to the case $W_2=Z$ and then apply
Theorem~\pref{T:Weilr} or argue by functorial characterizations and diagonals,
cf.~\cite[Props.~3.5 and 3.8]{rydh_hilbert}. The only exception is
``locally of finite
type''. In this case one first easily reduces to the situation where $W_1=W\to
W_2=Z$ is of finite type. Then $W\to Z$ has an approximation
by~\cite[Prop.~7.6 or Thm.~D]{rydh_noetherian-approx}. This means that
we can write $W\inj W_0\to Z$ with $W\inj W_0$ a closed immersion and $W_0\to Z$
of finite presentation and the result follows.

Similarly, if $Z\to S$ and $X\to S$ are as in Corollary~\pref{C:homstacks}
and in addition $X\to S$ has one of the properties: affine, quasi-affine,
Deligne--Mumford,
representable, representable and separated, locally of finite presentation,
locally of finite type, unramified, \etale; then so has $\Hom_S(Z,X)$.
\end{remark}

\begin{remark}[Smoothness]\label{R:weilr-smoothness}
If $W\to Z$ is smooth, then this does not imply that $f_*W\to S$ is smooth
unless $Z\to S$ is finite. The proof of~\cite[Prop.~3.5 (iv)]{rydh_hilbert}
does not apply since formal smoothness only implies that the infinitesimal
lifting property holds for thickenings of \emph{affine} schemes. For a
counter-example, let $Z\to S$ be a one-parameter family of twisted cubics
degenerating to a nodal plane curve with an embedded component. Then
$\Hom_S(Z,\mathbb{A}_S^1)=\weilr_{Z/S}(\mathbb{A}_Z^1)$ is an affine scheme
over $S$ with generic fiber $\mathbb{A}^1$ and special fiber $\mathbb{A}^2$,
hence not smooth. We thank the referee for making us aware of this fact.
\end{remark}

\begin{remark}[Boundedness]
If $W\to Z$ is quasi-compact, then it is non-trivial to show that $f_*W\to S$
is quasi-compact. Some results are available, however, such
as~\cite{olsson_boundedness-homstacks},
\cite[App.~C]{abramovich-olsson-vistoli_twisted-stable-maps-tame-stacks}
and~\cite[Prop.~3.8]{rydh_hilbert}.
These results imply corresponding boundedness results for
$\Hom$-stacks and have, for example, been used to deduce that the stack of twisted stable
maps has quasi-compact components under mild hypotheses.
\end{remark}

\begin{remark}
The proof of 
Theorem~\pref{T:Hilb-stack} also shows that if $X=\varprojlim_\lambda
X_\lambda$, where $\{ X_\lambda \}_{\lambda}$ is an inverse system of algebraic stacks of finite presentation over $S$ with affine bonding maps, then $\HilbSt^{\qfin}_{X/S}=\varprojlim_\lambda
\HilbSt^{\qfin}_{X_\lambda/S}$ and this inverse system has affine bonding maps.
\end{remark}

\end{section}


\begin{section}{Intrinsic finiteness for sheaves}
In this section we introduce the relative finiteness notion---intrinsically
of finite presentation---referred to in the introduction. This notion is needed
in the definition of the stack $\Coh(X/S)$ and the sheaf $\Quot(X/S,\sF)$ when
$X\to S$ is not locally of finite presentation. To motivate this definition,
note that if $\map{q}{Z}{X}$ is a finite morphism and $\map{p}{Z\to X}{S}$ is of
finite presentation, then $q_*\sO_Z$ is of finite type but not necessarily of
finite presentation. Conversely, if $q_*\sO_Z$ is of finite presentation, then
this does not imply that $p$ is of finite presentation. The new finiteness
notion fixes this: $q_*\sO_Z$ is intrinsically of finite presentation over $S$
exactly when $p$ is of finite presentation. Moreover, this notion is also
defined for sheaves of $\sO_X$-modules. We begin with the affine case.

\begin{definition}
Let $A$ be a ring, let $B$ be an $A$-algebra and let $M$ be a $B$-module. We
say that $M$ is \emph{intrinsically of finite presentation over $A$} if there
exist a polynomial ring $A[x_1,x_2,\dots,x_n]$ and a homomorphism
$A[x_1,x_2,\dots,x_n]\to B$ such that $M$ is of finite presentation as an
$A[x_1,x_2,\dots,x_n]$-module.
\end{definition}

Although quite natural, we have not been able to find this definition in the
literature except in the special case when $B$ is of finite
type~\cite[\spref{0659}]{stacks-project} under the name ``finitely presented
relative to $A$''.
The following lemma is of fundamental importance.

\begin{lemma}\label{L:Nakayama-like}
Let $A$ be a ring, let $B$ be an $A$-algebra and let $M$ be a
$B$-module.
\begin{enumerate}
\item If $M$ is finitely generated as an $A$-module, then $M$ is finitely
  generated as a $B$-module and $B/\Ann_B M$ is integral over $A$. In
  particular, the image of $\Supp_B M$ along $\Spec B\to \Spec A$ is
  $\Supp_A M$.
\item If $B$ is finitely generated as an $A$-algebra and $M$ is finitely
  presented as an $A$-module, then $M$ is finitely presented as a $B$-module.
\end{enumerate}
\end{lemma}
\begin{proof}
If $M$ is finitely generated as an $A$-module, then clearly $M$ is finitely
generated as a $B$-module. To see that $B/\Ann_B M$ is integral over $A$, we
may replace $B$ with $B/\Ann_B M$ so $B\to \End_A M$ becomes injective. Now
Cayley--Hamilton's theorem~\cite[Thm.~4.3]{eisenbud_comm_alg} shows that every
$b\in B$ satisfies an integral equation with coefficients in $A$.

To prove the second statement, choose a surjection
$A^n\surj M$ and note that the kernel $K$ is a finitely generated $A$-module.
The kernel $K_B$ of ${B^n\surj M\otimes_A B}$ is thus also finitely
generated. Let $L$ be
the kernel of the surjective homomorphism $B^n\surj M\otimes_A B\surj M$ and
let $N$ be the kernel of the surjective homomorphism $M\otimes_A B\to M$. Then
$L$ is an extension of $N$ by $K_B$ so it is enough to show that $N$ is a
finitely generated $B$-module. If $b_1,b_2,\dots,b_n$ are generators of $B$ as
an $A$-algebra and $m_1,m_2,\dots,m_r$ are generators of $M$ as an $A$-module,
then $m_j\otimes b_i-(b_im_j)\otimes 1$ are generators of $N$ as a $B$-module.
\end{proof}

Let $A$ be a ring, let $B$ be an $A$-algebra and let $M$ be a $B$-module.
Then from the previous lemma we obtain that:
\[
  \text{$M$ is i.f.p.\ over $A$} \;\implies \text{$M$ is f.g.\ as a $B$-module.}
\]
If $B$ is an $A$-algebra of finite type, then:
\[
  \text{$M$ is i.f.p.\ over $A$} \;\implies \text{$M$ is f.p.\ as a $B$-module}
\]
and the converse holds if $B$ is an $A$-algebra of finite presentation.
Finally, if $C$ is a $B$-algebra, then
\[
  \text{$C$ is i.f.p.\ over $A$} \iff \text{$C$ is f.g.\ as a $B$-module and
  f.p.\ as an $A$-algebra.}
\]
%
%
\begin{lemma}\label{L:ifp-char-limits}
If $B=\varinjlim_\lambda B_\lambda$ is a direct limit of $A$-algebras of
finite presentation and $M$ is a $B$-module, then the following are equivalent:
\begin{enumerate}
\item $M$ is intrinsically of finite presentation over $A$,
\item $M$ is a finitely presented $B_\lambda$-module for all sufficiently large
  $\lambda$,
\item $M$ is a finitely presented $B_\lambda$-module for some $\lambda$.
\end{enumerate}
\end{lemma}
\begin{proof}
This follows from the previous lemma and the observation that every homomorphism
$A[x_1,x_2,\dots,x_n]\to B$ factors through $B_\lambda$ for all sufficiently
large $\lambda$.
\end{proof}

From the characterization in Lemma~\pref{L:ifp-char-limits} we easily obtain that the
property ``intrinsically
of finite presentation over the base'' is stable under base change,
fpqc-local on the base, stable on the source under pull-back by finitely
presented morphisms and fppf-local on the source:
\begin{itemize}
\item Let $A'$ be an $A$-algebra. If $M$ is i.f.p.\
  over $A$, then ${M\otimes_A A'}$ is i.f.p.\ over~$A'$. The converse
  holds if $A\inj A'$ is faithfully flat.
\item Let $B'$ be a finitely presented $B$-algebra. If $M$ is i.f.p.\
  over $A$, then $M\otimes_B B'$ is i.f.p.\ over $A$. The converse holds
  if $B\inj B'$ is faithfully flat.
\end{itemize}
In particular, the property is fppf-local on source and target, so we may
extend the definition to algebraic stacks as follows.

\begin{definition}
Let $\map{f}{X}{S}$ be a morphism of algebraic stacks. A quasi-coherent
$\sO_X$-module $\sF$ is \emph{intrinsically of finite presentation over $S$} if
fppf-locally on $X$ and $S$, it is intrinsically of finite presentation. If, in
addition, $\Supp \sF$ is universally closed, quasi-compact and separated over
$S$, then we say
that $\sF$ is \emph{intrinsically proper over $S$}.
\end{definition}

\begin{proposition}\label{P:intrinsic-fp}
Let $\map{f}{X}{S}$ be a morphism of algebraic stacks. Let $\sF$ be a
quasi-coherent $\sO_X$-module.
\begin{enumerate}
\item\label{PI:ifp:fp}
  If $f$ is of finite presentation, then $\sF$ is intrinsically of finite
  presentation over $S$ (resp.\ intrinsically proper over $S$) if and only if
  $\sF$ is of finite presentation (resp.\ has proper support over $S$).
\item\label{PI:ifp:bc}
  Stability under base change: if $\sF$ is intrinsically of finite
  presentation over $S$ (resp.\ intrinsically proper over $S$) and $S'\to S$ is
  any morphism, then so is the base change $\sF'$ over $S'$.
\item\label{PI:ifp:limit}
  Suppose that $X=\varprojlim X_\lambda$ is the limit of an inverse
  system of finitely presented $S$-stacks $\{X_\lambda\to S\}$ with affine
  bonding maps $X_\lambda\to X_\mu$ and let $\map{h_\lambda}{X}{X_\lambda}$
  denote the canonical morphism. Then the following are equivalent:
\begin{enumerate}
\item $\sF$ is intrinsically of finite
  presentation over $S$.
\item There exists an index $\alpha$ such that
  $(h_\alpha)_*\sF$ is of finite presentation.
\item There exists an index $\alpha$ such that
  $(h_\lambda)_*\sF$ is of finite presentation for all $\lambda\geq \alpha$.
\end{enumerate}
%
%
\item\label{PI:ifp:limit-proper}
  Given an approximation $X\overset{h}{\to} X_0\to S$, with $X_0\to S$
  separated, the following are equivalent:
\begin{enumerate}
\item $\sF$ is intrinsically of finite
  presentation and intrinsically proper over~$S$.
\item $h_*\sF$ is of finite presentation with proper support over $S$.
\end{enumerate}
\end{enumerate}
\end{proposition}
\begin{proof}
  \ref{PI:ifp:fp}--\ref{PI:ifp:limit} follow directly from the affine case
  so it remains to prove~\ref{PI:ifp:limit-proper}.
  If $h_*\sF$ is of finite presentation, then $\sF$ is intrinsically
  of finite presentation by~\ref{PI:ifp:limit}. Also, $\sF$ is finitely
  generated
  and $\Supp \sF\inj X\to X_0$ is integral with image $\Supp h_*\sF$
  (Lemma~\ref{L:Nakayama-like}). It
  follows that $\sF$ is intrinsically of finite presentation and intrinsically
  proper over $S$.

  For the converse, we may work locally on $S$ and assume that $S$ is affine.
  Then $X_0$ is pseudo-noetherian so we may write $X=\varprojlim X_\lambda$
  where $\map{g_\lambda}{X_\lambda}{X_0}$ is affine of finite presentation.
  Let $h_\lambda$ denote the induced map $X\to X_\lambda$.  The push-forward
  $(h_\lambda)_*\sF$ is of finite presentation for sufficiently large $\lambda$
  by~\ref{PI:ifp:limit}.
  Let $Z_\lambda\inj X_\lambda$ denote the closed substack defined by
  the zeroth Fitting ideal of $(h_\lambda)_*\sF$. Since
  $(h_\lambda)_*\sF$ is finitely presented, the Fitting ideal is finitely
  generated, so $\injmap{j}{Z_\lambda}{X_\lambda}$ is of finite presentation.
  The Fitting ideal is contained in the annihilator of $(h_\lambda)_*\sF$ and
  contains a power of the annihilator~\cite[Prop.~20.7]{eisenbud_comm_alg}.
  Thus, $(h_\lambda)_*\sF=j_*j^*(h_\lambda)_*\sF$ and $|Z_\lambda|=\Supp \sF$.
  Moreover, since $X_0\to S$ is separated, it follows that $Z_\lambda\to X_0$
  is proper and affine, hence finite. We conclude that
  $h_*\sF=(g_\lambda)_*(h_\lambda)_*\sF$ is of finite presentation with
  proper support.
\end{proof}

\end{section}


\begin{section}{Stacks of sheaves}
In this section we prove Theorem~\tref{T:COH} and related
results on $\Hom$-spaces of sheaves.

\begin{definition}
Let $\map{f}{X}{S}$ be a separated morphism of algebraic stacks.
\begin{itemize}
\item The stack of coherent sheaves $\Coh(X/S)$ is the category with objects
  $(T,\sG)$ where $T$ is an $S$-scheme and $\sG$ is a quasi-coherent sheaf of
  $\sO_{X\times_S T}$-modules that is flat over $T$,
  intrinsically of finite presentation over
  $T$ and intrinsically proper over $T$.
\item The stack of coherent algebras $\CohAlg(X/S)$ is the analogous
  category with finite algebras instead of modules.
\item Let $\sF$ be a quasi-coherent $\sO_X$-module. The functor
  $\Quot(X/S,\sF)$ takes an $S$-scheme $T$ to the set of quotients
  $\sF_{X\times_S T}\surj
  \sG$ (up to isomorphism) such that $\sG$ is flat over $T$,
  intrinsically of finite
  presentation over $T$ and intrinsically proper over $T$.
\end{itemize}
\end{definition}

There are natural isomorphisms $\FHilb_{X/S}=\Quot(X/S,\sO_X)$ and
$\HilbSt^{\qfin}_{X/S}=\CohAlg(X/S)$ that take a family
$({\map{p}{Z}{T}},\map{q}{Z}{X})$ to the $\sO_{X\times_S T}$-module
$(q,p)_*\sO_Z$, noting that
$(q,p)$ is finite since $X\to S$ is separated.
Moreover, the natural forgetful morphism $\CohAlg(X/S)\to \Coh(X/S)$
is represented by affine morphisms. This follows as
in~\cite[Prop.~2.5]{lieblich_Coh_stack} and~\cite[Lem.~4.2]{rydh_hilbert}
using Theorem~\pref{T:Hom-sheaf} below. Thus, in the separated case,
Theorem~\tref{T:HILB} follows from Theorem~\tref{T:COH}.

\begin{theorem}\label{T:Hom-sheaf}
Let $\map{f}{X}{S}$ be a morphism of algebraic stacks
that is separated and locally of approximation type.
Let $\sF,\sG\in\QCoh(X)$ and assume that $\sG$ is flat, intrinsically of finite
presentation and intrinsically proper over $S$. Then $\Hom_{\sO_X/S}(\sF,\sG)$
is affine. If, in addition, $\sF$ is intrinsically of finite presentation,
then $\Hom_{\sO_X/S}(\sF,\sG)$ is of finite type.
\end{theorem}
\begin{proof}
The question is fppf-local on $S$ so we assume that $S$ is affine. We may
replace $X$ with the closed substack defined by $\Ann_{\sO_X} \sG$ and
assume that $X\to S$ is quasi-compact and universally closed.
After replacing $S$ with an fppf-covering, we may
then assume that $X\to S$ has an approximation $X\overset{h}{\to} X_0\to S$ with
$X_0\to S$ separated.
Then $h_*\sG$ is of finite presentation with proper support over $S$.

The natural morphism $\Hom_{\sO_X/S}(\sF,\sG)\to
\Hom_{\sO_{X_0}/S}\bigl(h_*\sF,h_*\sG\bigr)$ is
a monomorphism since $h$ is affine. As we will see, this monomorphism
is represented by closed immersions. Since
$\Hom_{\sO_{X_0}/S}\bigl(h_*\sF,h_*\sG\bigr)$ is
affine~\cite[Thm.~D]{hall_coho-bc} this will prove that
$\Hom_{\sO_X/S}(\sF,\sG)$ is affine.

Let $\sK$ be the kernel of the surjection $h^*h_*\sF\to
\sF$. A homomorphism $h_*\sF\to h_*\sG$ induces a
homomorphism $h^*h_*\sF\to \sG$, which factors uniquely
through $\sF$ if and only if the composition $\sK\to \sG$ is zero. This happens
if and only if $h_*\sK\to h_*\sG$ is zero. This is a closed
condition since $\Hom_{\sO_{X_0}/S}\bigl(h_*\sK,h_*\sG\bigr)$ is
separated (even affine by~\cite[Thm.~D]{hall_coho-bc}).

If $\sF$ is also intrinsically of finite presentation, then $h_*\sF$ is of
finite presentation.
It follows that $\Hom_{\sO_{X_0}/S}\bigl(h_*\sF,h_*\sG\bigr)$ is
of finite presentation and we conclude that $\Hom_{\sO_X/S}(\sF,\sG)$ is of
finite type.
\end{proof}

\begin{corollary}\label{C:weilr-of-modules}
Let $\map{f}{X}{S}$ be a morphism of algebraic stacks that is separated and locally of approximation type. Let $\sA$ be a quasi-coherent $\sO_X$-algebra. Let $\sG\in\QCoh(X)$
be flat, intrinsically of finite presentation and intrinsically
proper over $S$. Then the sheaf $\weilr_{\sG/S}(\sA)$, which takes a morphism
$\map{h}{T}{S}$ to the set of $h^*\sA$-module structures on $h^*\sG$, is
affine over $S$.
\end{corollary}
Note that $\weilr_{\sG/S}(\sA)
=\weilr_{\Spec \sG\to S}(\Spec(\sA\otimes_{\sO_X} \sG)/\Spec \sG)$
when $\sG$ is a quotient sheaf of $\sO_X$, explaining the notation.
Corollary~\pref{C:weilr-of-modules}
generalizes~\cite[Thm.~3.5]{skjelnes_Weil-restriction-of-mod}: if
$\map{f}{X}{S}$ is \emph{affine}, then $\sG\in\QCoh(X)$ is intrinsically of
finite presentation and intrinsically proper over $S$ if and only if $f_*\sG$ is
of finite presentation (Proposition~\ref{P:intrinsic-fp}).
Thus $\weilr_{\sG/S}(\sA)$ equals the module
restriction functor $\mathscr{M}od^M_{B\to R}$ and we
recover~\cite[Thm.~3.5]{skjelnes_Weil-restriction-of-mod}.
\begin{proof}[Proof of Corollary~\pref{C:weilr-of-modules}]
The question is local on $S$ so we can assume that $S$ is affine and $X\to S$
is quasi-compact and admits an approximation $X\to X_0\to S$.

Consider the functor $\Hom_{\sO_X/S}(\sA\otimes_{\sO_X} \sG,\sG)$, which is an
affine
$S$-scheme by Theorem~\pref{T:Hom-sheaf}. Let $\map{\varphi}{\sA\otimes_{\sO_X}
  \sG}{\sG}$ denote the universal homomorphism (after replacing $S$ with the
$\Hom$-space). The Weil restriction $\weilr_{\sG/S}(\sA)$ is then the 
subfunctor given by the conditions that the maps:
\begin{align*}
\id{\sG}  - \varphi\circ (\eta\otimes \id{\sG}) \colon
  & \sG \to \sG \\
\varphi\circ (\mu\otimes \id{\sG}) - \varphi\circ (\id{\sA} \otimes \varphi)
\colon
  & \sA\otimes_{\sO_X} \sA\otimes_{\sO_X} \sG \to \sG
\end{align*}
vanish, where $\map{\eta}{\sO_X}{\sA}$ is the unit and
$\map{\mu}{\sA\otimes_{\sO_X}\sA}{\sA}$ is the multiplication.
This is a closed subfunctor
since $\Hom_{\sO_X/S}(\sF,\sG)$ is affine for any quasi-coherent $\sO_X$-module
$\sF$.
\end{proof}

\begin{theorem}\label{T:Coh-stack}
Let $\map{f}{X}{S}$ be a morphism of algebraic stacks
that is separated and locally of approximation type. The stack
$\Coh(X/S)$ is algebraic with affine diagonal. If $\sF\in\QCoh(X)$,
then $\Quot(X/S,\sF)$ is a separated algebraic space.
\end{theorem}
\begin{proof}
We argue almost exactly as in the proof of Theorem~\pref{T:Hilb-stack}. First
we reduce to the case where $S$ is affine and $X$ quasi-compact. Next, we further reduce to the case where there is
an approximation $X\overset{h}{\to} X_0\to S$. Then there is a natural morphism
$\map{h_*}{\Coh(X/S)}{\Coh(X_0/S)}$ that takes a sheaf $\sG$ to $h_*\sG$. The
stack $\Coh(X_0/S)$ is algebraic, locally of finite presentation over $S$ and
has affine diagonal~\cite[Thm.~8.1]{hall_coherent-versality}. The morphism
$h_*$ is represented by affine morphisms: given a morphism $T\to \Coh(X_0/S)$,
corresponding to a finitely presented sheaf $\sG$ on $X_0\times_S T$, the
liftings to $\Coh(X/S)$ correspond to the $h_*\sO_X$-module structures on
$\sG$. Thus, $h_*$ is represented by $\weilr_{\sG/T}(h_*\sO_X)$ which is
affine by Corollary~\pref{C:weilr-of-modules}.
%
\end{proof}

If $f$ has finite diagonal, then $f$ is locally of approximation type.
Thus, Theorem~\tref{T:COH} is an immediate
consequence of Theorem~\pref{T:Coh-stack}.

\end{section}


\bibliography{gen-hilb-quot}
\bibliographystyle{dary}

\end{document}